\documentclass{amsart}
\usepackage{amssymb}

\title{Covering Numbers in Linear Algebra}
\author{Pete L. Clark}

\begin{document}

\maketitle

\newtheorem{lemma}{Lemma}
\newtheorem{prop}[lemma]{Proposition}
\newtheorem{cor}[lemma]{Corollary}
\newtheorem{thm}[lemma]{Theorem}
\newtheorem{example}[lemma]{Example}
\newtheorem{thm?}[lemma]{Theorem?}
\newtheorem{schol}[lemma]{Scholium}
\newtheorem{ques}{Question}
\newtheorem{conj}[lemma]{Conjecture}
\newtheorem*{mainthm}{Main Theorem}
\newtheorem{prob}{Problem}

\newcommand{\Z}{\mathbb{Z}}
\newcommand{\R}{\mathbb{R}}

\renewcommand{\dim}{\operatorname{dim}}
\newcommand{\LC}{\operatorname{LC}}
\newcommand{\AC}{\operatorname{AC}}
\newcommand{\ILC}{\operatorname{ILC}}
\newcommand{\IAC}{\operatorname{IAC}}

\begin{abstract}
We compute the minimal cardinality of coverings and irredundant coverings of a vector space over an arbitrary 
field by proper linear subspaces.  Analogues for affine linear subspaces are also given.  
\end{abstract}

\noindent
Notation: The cardinality of a set $S$ will be denoted by $\# S$.  For a vector space $V$ over a field $K$, we denote its dimension -- i.e., the cardinality of any $K$-basis of $V$ -- 
by $\dim K$.  We emphasize that $\dim V$ is a possibly infinite cardinal number: infinite-dimensional vector spaces are allowed. 

\section{Linear coverings}
\noindent
Let $V$ be a vector space over a field $K$.  A \textbf{linear covering} 
of $V$ is a collection $\{W_i\}_{i \in I}$ of proper $K$-subspaces such that $V = \bigcup_{i \in I} W_i$.  A linear 
covering is \textbf{irredundant} if for all $J \subsetneq I$, $\bigcup_{i \in J} W_i \neq V$.  Linear coverings exist if and only if $\dim V \geq 2$.  
\\ \\
The \textbf{linear covering number} $\LC(V)$ of a vector space $V$ of dimension at least $2$ is the least cardinality $\# I$ of a linear covering $\{W_i\}_{i \in I}$ of $V$.  The \textbf{irredundant linear covering number} $\ILC(V)$ is 
the least cardinality of an irredundant linear covering of $V$.  Thus $\LC(V) \leq \ILC(V)$.  
\\ \\
The main result of this note is a computation of $\LC(V)$ and $\ILC(V)$.  
\begin{mainthm} Let $V$ be a vector space over a field $K$, with $\dim V \geq 2$. \\
a) If $\dim V$ and $\#K$ are not both infinite, then $\LC(V) = \# K  +1$.  \\
b) If $\dim V$ and $\# K$ are both infinite, then $\LC(V) = \aleph_0$. \\
c) In all cases we have $\ILC(V) = \# K + 1$. 
\end{mainthm}
\noindent
Here is a counterintuitive consequence: the vector space $\R[t]$ of polynomials has a countably infinite linear 
covering -- indeed, for each $n \in \Z^+$, let $W_n$ be the subspace of polynomials of degree at most $n$.  However any 
irredundant linear covering of $\R[t]$ has cardinality $\# \R + 1 = 2^{\aleph_0}$.  Redundant coverings can be much more 
efficient!
\\ \\
The fact that a finite-dimensional vector space over an infinite field cannot be a finite union of proper linear subspaces 
is part of the mathematical folkore: the problem and its solution appear many times in the literature.  For instance problem 10707 in this \textsc{Monthly} is intermediate between this fact and our main result.  The editorial comments given on page 951 of the December 2000 issue of the \textsc{Monthly} give references to variants of this fact dating back to $1959$.  Like many pieces of folklore, there seems to a be mild stigma against putting it in standard texts; an exception is \cite[Thm. 1.2]{Roman}.  
\\ \indent
There are two essentially different arguments that establish this fact.  Upon examination, each of these yields a stronger result, recorded as Theorem \ref{THMA} and Theorem \ref{THMB} below.  From these two results the Main Theorem follows easily. \\ \indent The first two parts of the Main Theorem have appeared in the literature before (but only very recently!): they were shown by 
A. Khare \cite{Khare1}.  I found these results independently in the summer of 2008.  The computation of the irredundant linear covering number appears to be new.  The proof of the Main Theorem is given in $\S 2$.  \\ \indent
There is an analogous result for coverings of a vector space by affine linear subspaces.  This is stated in $\S 3$; we then briefly discuss what modifications must be made in the proof of the Main Theorem to obtain this affine analogue.  
\section{Proof of the Main Theorem}
\noindent
First we prove three lemmas, all consequences or special cases of the Main Theorem.  

\begin{lemma}(Quotient Principle)
\label{QUOTIENTPRINCIPLE}
\label{LEMMA0}
Let $V$ and $W$ be vector spaces over a field $K$ with $\dim V \geq \dim W \geq 2$.  Then $\LC(V) \leq \LC(W)$ and 
$\ILC(V) \leq \ILC(W)$.  
\end{lemma}
\begin{proof}
By standard linear algebra, the hypothesis implies that there is a surjective linear map $q: V \rightarrow W$.  If 
$\{W_i\}_{i \in I}$ is a linear covering of $W$, then the complete preimages $\{q^{-1}(W_i)\}_{i \in I}$ give a linear 
covering of $V$.  The preimage of an irredundant covering is easily seen to be irredundant.  
\end{proof}

\begin{lemma}
\label{LEMMA1}
For any field $K$, the unique linear covering of $K^2$ is the set of all lines through the origin, of cardinality $\# K + 1$.  It is an irredundant covering.  
\end{lemma}
\begin{proof} The set of lines through the origin is a linear covering of $K^2$.  Moreover, any nonzero $v \in K^2$ lies on a unique line, so all lines are needed.  The lines through the origin are 
$\{y = \alpha x \ | \ \alpha \in K\}$ and $x = 0$, so there are $\# K + 1$ of them.
\end{proof} 

\begin{lemma}
\label{LEMMA1.5}
Let $V$ be a vector space over a field $K$, with $\dim V \geq 2$.  Then there are at least $\# K + 1$ hyperplanes -- 
i.e., codimension-one linear subspaces -- in $V$.
\end{lemma}
\begin{proof} When $\dim V = 2$ there are exactly $\# K + 1$ hyperplanes $\{L_i\}$.  In general, take a surjective linear map $q: V \rightarrow K^2$; then $\{q^{-1}(L_i)\}$ is a family of distinct hyperplanes in $V$.
\end{proof}
\noindent
The exact number of hyperplanes in $V$ is of course known, but not needed here.

\begin{thm}
\label{THMA}
Let $V$ be a finite-dimensional vector space over a field $K$, and let $\{W_i\}_{i \in I}$ be a linear covering of 
$V$.  Then $\# I \geq \# K + 1$.
\end{thm}
\begin{proof} 
Since every proper subspace is contained in a hyperplane, it suffices to consider hyperplane coverings.  We go by induction on $d$, the case $d = 2$ being Lemma \ref{LEMMA1}.  Assume the result for $(d-1)$-dimensional spaces and, 
seeking a contradiction, that we have a linear covering $\{W_i\}_{i \in I}$ of $K^d$ with $\# I < \# K + 1$.  By 
Lemma \ref{LEMMA1.5}, there is a hyperplane $W$ such that $W \neq W_i$ for any $i \in I$.  Then $\{W_i \cap W\}_{i \in I}$ is a covering of $W \cong K^{d-1}$ by at most $\# I < \# K +1$ hyperplanes, giving a contradiction.
\end{proof}

\begin{thm}
\label{THMB}
Let $V$ be a vector space over a field $K$, and let $\{W_i\}_{i \in I}$ be an irredundant linear covering of $V$.  
Then $\# I \geq \# K + 1$.
\end{thm}
\begin{proof} 
Let $\{W_i\}_{i \in I}$ be an irredundant linear covering of $V$.  Choose one of the subspaces 
in the covering, say $W_{\bullet}$.  By irredundancy, there exists $u \in W_{\bullet} \setminus \bigcup_{i \neq \bullet} W_i$; certainly there exists $v \in V \setminus W_{\bullet}$.  Consider the affine line $\ell = \{tu + v \ |  \ t \ \in K\}$; evidently $\# \ell = \# K$.  If $w  = tu + v \in \ell \cap W_{\bullet}$, then $v = w - tu \in W_{\bullet}$, giving a contradiction.  Further, if for any $i \neq \bullet$ we had $\# (\ell \cap W_i) \geq 2$, then we would have $\ell \subset 
W_i$ and thus also the $K$-span of $\ell$ is contained in $W_i$, so $u = (u+v) - v) \in W_i$, again giving a contradiction.  
It follows that $\# \ell = \# K \leq \# (I \setminus \{\bullet \})$.
\end{proof}
\noindent
\emph{Proof of the Main Theorem}: \\ \\
Let $V$ be a $K$-vector space of dimension at least $2$. By Theorem \ref{THMB}, $\ILC(V) \geq \#K + 1$, and by Lemmas \ref{LEMMA0} and \ref{LEMMA1}, $\ILC(V) \leq \# K + 1$.  So $\ILC(V) = \#K +1$, proving part c) of the Main Theorem.  It remains to compute $\LC(V)$. 
\\ 
Case 1: Suppose $2 \leq \dim V < \aleph_0$.  By Theorem \ref{THMA} we have $\LC(V) \geq \# K + 1$, 
whereas by Lemma 1 and Lemma 2 we have $\LC(V) \leq \LC(K^2) = \# K + 1$.
\\ 
Case 2: Suppose $\dim V \geq \aleph_0$ and $K$ is finite.  Then $\LC(V) \leq \LC(K^2) = \# K + 1 < \aleph_0$.  Suppose that $V$ had a linear covering $\{W_i\}_{i=1}^n$ 
with $n < \# K + 1$.  Then, since $n$ is finite, we may obtain an irredundant subcovering simply by removing 
redundant subspaces one at a time, until we get an irredundant covering by $m$ subspaces, with $m \leq n < \#K + 1$ subspaces, 
contradicting Theorem \ref{THMB}.
\\ 
Case 3: Suppose $\dim V$ and $\# K$ are both infinite.  Consider $W = \bigoplus_{i=1}^{\infty} K$, a vector 
space of dimension $\aleph_0$.  For $n \in \Z^+$, put $W_n := \bigoplus_{i=1}^n K$.  Then $\{W_n\}_{n=1}^{\infty}$ gives a covering of $W$ of cardinality $\aleph_0$.  Since $\dim V \geq \dim W$, by Lemma 4 we have $\LC(V) \leq \aleph_0$.  Thus it 
remains to show that $V$ does not admit a finite linear covering.  But once again, if $V$ admitted a finite 
linear covering it would admit a finite irredundant linear covering, contradicting Theorem \ref{THMB}.

\section{Affine Covering Numbers}
\noindent
An \textbf{affine covering} $\{A_i\}_{i \in I}$ of a vector space $V$ is a covering by translates of proper linear subspaces.  An affine covering is irredundant if no proper subset gives a covering.  Irredundant affine coverings exist if and only if $\dim V \geq 1$.  The \textbf{affine covering number} $\AC(V)$ is the least cardinality of an affine covering, 
and similarly the \textbf{irredundant affine covering number} $\IAC(V)$ is the least cardinality of an irredundant affine 
covering.  
\begin{thm}
Let $V$ be a vector space over a field $K$, with $\dim V \geq 1$. \\
a) If $\min(\dim V, \# K)$ is finite, then $\AC(V) = \# K$. \\
b) If $\dim V$ and $\# K$ are both infinite, then $\AC(V) = \aleph_0$. \\
c) We have $\IAC(V) = \# K$.
\end{thm}
\noindent
The proof of the Main Theorem goes through with minor modifications.  Lemma \ref{QUOTIENTPRINCIPLE} holds verbatim.  The following self-evident result is the analogue of Lemma \ref{LEMMA1}.
\begin{lemma}
\label{LEMMA1AFFINE}
For any field $K$, the unique affine covering of $K^1$ is the set of all points of $K$, of cardinality $\# K$.  It is 
an irredundant covering.
\end{lemma}
\noindent 
Combining these two results we get the analogue of Lemma \ref{LEMMA1.5}, in which $\# K + 1$ is replaced by $\# K$.  
To prove the analogue of Theorem \ref{THMA}, note that for two codimension-one affine subspaces $W_1$, $W_2$ of a vector space $V$, $W_1 \cap W_2$ is either empty or is a codimension-one affine subspace in each $W_i$.  
In the proof of the analogue of Theorem \ref{THMB} we use the line $\ell = \{(1-t) u + tv \ | \ t \in K\}$.
\\ \\
\textbf{Acknowledgments.}  This work was partially supported by National Science Foundation grant DMS-0701771.

\noindent
\emph{Department of Mathematics, University 
of Georgia, Athens, GA 30602-7403}
\\ 
\emph{pete@math.uga.edu}

\end{document}